\documentclass{amsart}
\usepackage{amsfonts,amssymb,amsmath,amsthm,bbm,xcolor,ulem}
\newtheorem{theorem}{Theorem}[section]
\newtheorem{lemma}[theorem]{Lemma}

\newtheorem{corollary}[theorem]{Corollary}

\newtheorem{remark}[theorem]{Remark}
\newtheorem{conjecture}[theorem]{Conjecture}
\theoremstyle{definition}
\numberwithin{equation}{section}

\def\a{\mathfrak{a}}

\def\R{{\bf R}}

\def\C{{\bf C}}

\def\al{{\alpha_\lambda}}
\def\ax{{\alpha_X}}
\def\apx{{\alpha_{X^+}}}
\def\bl{{\beta_\lambda}}
\def\bx{{\beta_X}}
\def\bpx{{\beta_{X^+}}}
\def\gl{{\gamma_\lambda}}
\def\gx{{\gamma_X}}
\def\gpx{{\gamma_{X^+}}}

\def\ll{{\lambda_2}}
\def\lll{{\lambda_3}}

\def\to{\rightarrow}

\begin{document}

\begin{center}
{\bf A formula and sharp estimates  for the  Dunkl    kernel  for the root system $A_2$}\\
 P. Graczyk, P. Sawyer\\[10mm]
\end{center}

\section*{Abstract}
In this paper, we transform a formula for the $A_2$ Dunkl kernel by B\'echir Amri.  The resulting formula expresses the $A_2$ Dunkl kernel in terms of the $A_1$ Dunkl kernel involving only positive terms.  
This result allows us to derive sharp estimates for the $A_2$ Dunkl kernel.  As an interesting by-product, we obtain sharp esitmates for the corresponding heat kernel.

\section{Introduction}

Estimates of Dunkl kernel and of Dunkl heat kernel are a challenging problem. Sharp estimates of these kernels were proven  only in rank 1, see \cite{AnkerDT}.

In the  Dunkl setting,  estimates of the Dunkl kernel $E_k(X,\lambda)$ are equivalent to
heat kernel estimates (refer to formula \eqref{heatR}).

A function $k: \Sigma \to \R$ is called a multiplicity function if it is invariant under the action of $W$ on $\Sigma.$
	
Let $\partial_\xi$ be the derivative in the direction of $\xi\in\R^d$.   The Dunkl operators indexed by $\xi$ are then given by
\begin{align*}
T_\xi(k)\,f(X)&=\partial_\xi\,f(X)+\sum_{\alpha\in \Sigma_+}\,k(\alpha)\,\alpha(\xi)\,\frac{f(X)-f(\sigma_\alpha\,X)}{\langle \alpha,X\rangle}.
\end{align*}

The $T_\xi$'s, $\xi\in\R^d$, form a commutative family.

For fixed $ Y\in\R^d$, the Dunkl kernel $E_k(\cdot,\cdot)$ is then the only real-analytic solution to the system
\begin{align*}
\left.T_\xi(k)\right\vert_X\,E_k(X,Y)
=\langle \xi,Y\rangle\, E_k(X,Y),~\forall\xi\in\R^d
\end{align*}
with $E_k(0,Y)=1$.  In fact, $E_k$ extends to a holomorphic function  on $\C^d\times \C^d$.

Up to now, the best  estimates of the Dunkl heat kernel for all root systems  were recently proven by Dziuba\'nski and Hejna \cite{DH}. They improve earlier 
results obtained in \cite{AnkerDH}. The results of  \cite{DH}
imply immediately the estimates of the Dunkl kernel 
$E_k(X,\lambda)$ with the exponential factor of the form 
$$
e^{c\langle \lambda^+, X^+\rangle}
$$
for any $c>1$ in the upper estimate and
any $c<1$ in the lower estimate
(here $Z^+$ denotes the projection of $Z$
on the positive Weyl chamber   $\overline{C^+}$, i.e. the value of $wZ$ with $w\in W$ chosen in such a way that $wZ\in \overline{C^+}$).
Our  recent results \cite{PGPS0}  in the $W$-invariant case for the system $A_n$
and the paper \cite{AnkerDT} devoted to the $A_1$ case suggest that sharp estimates of $E_k(X,\lambda)$ with the exponential factor 
$
e^{\langle \lambda^+, X^+\rangle}
$ hold true.

It is known since Dunkl's paper
\cite{Dunkl} that the root system $A_2$
allows some explicit or integral expressions for important objects of Dunkl analysis,
in particular for the Dunkl kernel 
$E_k(X,\lambda)$. Some integral formulas for $E_k(X,\lambda)$ of the system $A_2$ were given by Amri \cite{Amri01}.

In this paper, starting from a result by Amri (\cite[Theorem 1.1]{Amri01}),  we  prove a new integral formula for the Dunkl kernel 
$E_k(X,\lambda)$ (Theorem \ref{A2A1}) on $A_2$, relating it to the Dunkl kernel in rank 1.

This formula allows us to prove sharp estimates of  the Dunkl kernel 
$E_k(X,\lambda)$  on $A_2$ (Theorem \ref{sharp})  with the exponential factor $e^{\langle \lambda^+, X^+\rangle}$. The proof uses  techniques  analogous to those that we used in \cite{PGPS0}.

\begin{remark}
In 2017, Jacek Dziuba\'nski presented an upper estimate in a Weyl chamber
(\cite{Dtalk2017}) of a similar form to the ones we have
obtained for the Dunkl estimate (based on work by Jean-Philippe Anker and Bartosz Trojan).  In June 2023, Anker and Trojan informed us that they continued their study
and showed us their current achievement containing sharp lower and upper estimates
for the Dunkl kernel in the case of root systems $A_2$ and $B_2$.
Their method is completely different from ours.
\end{remark}

{\bf Acknowledgents.}
We thank J.-P. Anker, J. Dziuba\'nski, A. Hejna and B. Trojan for communications,  discussions and remarks that 
improved and enriched this paper.

\section{Basic notation and definitions}

We consider the root system $A_2$ on the space $\a=\{X=(x_1,x_2,x_3)\in\R^3\colon x_1+x_2+x_3=0\}$.  The Weyl group is $S_3$ acting as permutations on the $x_i$'s.  The positive Weyl chamber is $C^+=\{(x_1,x_2,x_3)\in \a\colon x_1>x_2>x_3\}$.

Given a domain $D$, $f(x)\asymp g(x)$ means that there exists $C>0$ such that $C^{-1}\,g(x)\leq f(x)\leq C\,g(x)$ for all $x\in D$.  Similarly, $f(x)\lesssim g(x)$ ($f(x)\gtrsim g(x)$) means that there exists $C>0$ such that $ f(x)\leq C\,g(x)$ ($ f(x)\geq C\,g(x)$) for all $x\in D$.

For $X=(x_1,x_2,x_3)$, we denote by   $X^+$ the projection of $X$ on $C^+$, i.e. the reordering of the entries of $X$ in decreasing order. For example, if $X$ is such that $x_2\ge x_1 \ge x_3$, then $X^+=(x_2, x_1, x_3)$.

\section{A new  formula for the Dunkl kernel on $A_2$}

The Dunkl kernel $E_k^{\text{rk\,1}}$ for the system $A_1$
has the integral representation \cite{Anker}
\begin{equation} \label{rank1int}
 E_k^{\text{rk\,1}}(x,v)=
 \frac{\Gamma(k+\frac12)}{\sqrt{\pi}\Gamma(k)}
 \int_{-1}^1  e^{zxv}(1-z)^{k-1} (1+z)^k dz=
 e^{vx}\frac{\Gamma(2k+1)}{k\Gamma(k)^2} \int_0^1 e^{-2xvz}z^{k-1}(1-z)^kdz
\end{equation}

Let
$\mathcal{J}_a$ denote the modified
Bessel function of index $a$.
Then
\begin{align}\label{rank1J}
E_k^{\text{rk\,1}}( x,v)=\mathcal{J}_{k-1/2}(vx)+\frac{vx}{2k+1}\mathcal{J}_{k+1/2}(v x).
\qquad\text{\cite[Remark 3.9]{Anker}}
\end{align}
Let $X=(x_1,x_2,x_3)$ with $x_1+x_2+x_3=0$.
We denote the positive roots 
of the system $A_2$ by $\alpha(X)=x_1-x_2$, $\beta(X)=x_2-x_3$ and $\gamma=\alpha+\beta$. For simplicity, we write $\ax=\alpha(X)$ and similarly
for
$\bx,\gx,\al,\bl,\gl$.

Let $W_k(\lambda,\nu)=( (\lambda_1-\nu_1)(\lambda_1-\nu_2)(\nu_1-\lambda_2)(\lambda_2-\nu_2)(\nu_1-\lambda_3)(\nu_2-\lambda_3))^{k-1}$.

In the next theorem
we give a new integral formula relating the
Dunkl kernel $E_k(X,\lambda)$ on $A_2$
with the 
Dunkl kernel $E_k^{\text{rk\,1}}$ on $A_1$.

\begin{theorem}
\label{A2A1}  Let $\lambda\in C^+$, the positive Weyl chamber.   
We have for any $X{\in\a}$,
\begin{align}\label{Amri1}
E_k(X,\lambda)&=\frac{3\,\Gamma(3k)}{V(\lambda)^{2k}\Gamma(k)^2}
\,\int_{\lambda_3}^{\lambda_2} \int_{\lambda_2}^{\lambda_1} 
\{
(\nu_1-\lambda_2)(\lambda_1-\nu_2)\,E_k^{\text{rk\,1}}((x_1-x_2)/2,(\nu_1-\nu_2))
\\ &\qquad\qquad\qquad\nonumber
+(\lambda_1-\nu_1)(\lambda_2-\nu_2)\,E_k^{\text{rk\,1}}(-(x_1-x_2)/2,(\nu_1-\nu_2))
\}
\\ &\qquad\qquad\qquad\nonumber
\,(\nu_1-\lambda_3) (\nu_2-\lambda_3)
e^{(x_1+x_2-2x_3)(\nu_1+\nu_2)/2}
\,W_k(\lambda,\nu),d\nu_1d\nu_2
\end{align}
and
\begin{align}\label{Amri2}
E_k(X,\lambda)&=\frac{3\,\Gamma(3k)}{V(\lambda)^{2k}\Gamma(k)^2}
\,\int_{\lambda_3}^{\lambda_2} \int_{\lambda_2}^{\lambda_1} 
\{
(\nu_1-\lambda_3)(\lambda_2-\nu_2)\,E_k^{\text{rk\,1}}((x_2-x_3)/2,(\nu_1-\nu_2))
\\ &\qquad\qquad\qquad\nonumber
+(\nu_1-\lambda_2)(\nu_2-\lambda_3)\,E_k^{\text{rk\,1}}(-(x_2-x_3)/2,(\nu_1-\nu_2))
\}
\\ &\qquad\qquad\qquad \nonumber
\,(\lambda_1-\nu_1) (\lambda_1-\nu_2)
e^{(x_2+x_3-2x_1)(\nu_1+\nu_2)/2}
\,W_k(\lambda,\nu),d\nu_1d\nu_2.
\end{align}

\end{theorem}

\begin{proof}
By the integral formula for $E_k$ of Amri \cite[Theorem 1.1]{Amri01} and the fact that $-2(\nu_1\nu_2+(\lambda_3/2)(\nu_1+\nu_2)+\lambda_1\lambda_2)=(\lambda_1-\lambda_2)(\nu_1-\nu_2) -2(\lambda_1-\nu_1)(\lambda_2-\nu_2)$ when 
$\lambda_1+\lambda_2+\lambda_3=0$, we find
\begin{align*}
E_k(X,\lambda)
&=\frac{3\,\Gamma(3k)}{V(\lambda)^{2k}\Gamma(k)^2}
\,\int_{\lambda_3}^{\lambda_2} \int_{\lambda_2}^{\lambda_1} 
\left\{(\lambda_1-\lambda_2)(\nu_1-\nu_2)\mathcal{J}_{k-1/2}\left(\frac{(x_1-x_2)(\nu_1-\nu_2)}{2}\right)
\right.\\&\qquad\qquad\left.
+[(\lambda_1-\lambda_2)(\nu_1-\nu_2)-2(\lambda_1-\nu_1)(\lambda_2-\nu_2)]\mathcal{J}'_{k-1/2}\left(\frac{(x_1-x_2)(\nu_1-\nu_2)}{2}\right)
\right\}
\\ &\qquad\qquad\qquad\nonumber
\,(\nu_1-\lambda_3) (\nu_2-\lambda_3)
e^{(x_1+x_2-2x_3)(\nu_1+\nu_2)/2}
\,W_k(\lambda,\nu),d\nu_1d\nu_2
\end{align*}

Noting that $\mathcal{J}'_a(x)=(x/(2\,(a+1)))\,\mathcal{J}_{a+1}(x)$ (\cite[(8)]{Amri01}) and using the formula 
\eqref{rank1J},
we obtain the following expression for
the term  $\{\dots\}$ in the last integral:
\begin{align*}
\{\dots\}
&=(\lambda_1-\lambda_2)(\nu_1-\nu_2)E_k^{\text{rk\,1}}((x_1-x_2)/2,(\nu_1-\nu_2))
\\&\qquad\qquad
-(\lambda_1-\nu_1)(\lambda_2-\nu_2)\frac{(x_1-x_2)(\nu_1-\nu_2)}{2\,k+1}\mathcal{J}_{k+1/2}\left(\frac{(x_1-x_2)(\nu_1-\nu_2)}{2}\right).
\end{align*}

Now, using the integral representation  of the modified Bessel function and
integration by parts, we get
\begin{align*}
\lefteqn{-\frac{(x_1-x_2)(\nu_1-\nu_2)}{2\,k+1}\mathcal{J}_{k+1/2}\left(\frac{(x_1-x_2)(\nu_1-\nu_2)}{2}\right)}\\
&=-2\,\frac{\Gamma(k+1/2)}{\sqrt{\pi}\Gamma(k)}\,\frac{(x_1-x_2)(\nu_1-\nu_2)}{4\,k}\,\int_{-1}^1\,e^{(x_1-x_2)(\nu_1-\nu_2)z/2}\,(1-z^2)^k\,dz\\
&=-2\,\frac{\Gamma(k+1/2)}{\sqrt{\pi}\Gamma(k)}\,\int_{-1}^1\,e^{(x_1-x_2)(\nu_1-\nu_2)z/2}\,z(1-z^2)^{k-1}\,dz\\
&=-2\,\frac{\Gamma(k+1/2)}{\sqrt{\pi}\Gamma(k)}\,\int_{-1}^1\,e^{(x_1-x_2)(\nu_1-\nu_2)z/2}\,(1+z-1)(1-z^2)^{k-1}\,dz\\
&=-2\,\frac{\Gamma(k+1/2)}{\sqrt{\pi}\Gamma(k)}\,\int_{-1}^1\,e^{(x_1-x_2)(\nu_1-\nu_2)z/2}\,(1+z)^k(1-z)^{k-1}\,dz
\\&\qquad
+2\,\frac{\Gamma(k+1/2)}{\sqrt{\pi}\Gamma(k)}\,\int_{-1}^1\,e^{(x_1-x_2)(\nu_1-\nu_2)z/2}\,(1-z^2)^{k-1}\,dz\\
&=-E_k^{\text{rk\,1}}((x_1-x_2)/2,(\nu_1-\nu_2))
\\&\qquad
+\frac{\Gamma(k+1/2)}{\sqrt{\pi}\Gamma(k)}\,\int_{-1}^1\,e^{(x_1-x_2)(\nu_1-\nu_2)z/2}\,\overbrace{[2 (1-z^2)^{k-1}-(1+z)^k(1-z)^{k-1}]}^{(1+z)^{k-1}(1-z)^{k}}\,dz\\
&=-E_k^{\text{rk\,1}}((x_1-x_2)/2,(\nu_1-\nu_2))+E_k^{\text{rk\,1}}(-(x_1-x_2)/2,(\nu_1-\nu_2))
\end{align*}
(in the last equalities we used
\eqref{rank1int}).
The  formula \eqref{Amri1} follows.  Interchanging the roles of the roots $\alpha$ and $\beta$ in \eqref{Amri1}, we obtain \eqref{Amri2}.
\end{proof}

\begin{remark}
It is important to note that in the formulas \eqref{Amri1} and \eqref{Amri2}, all terms are positive (the formula of Amri  \cite[Theorem 1.1]{Amri01} has a difference of two  terms).
Taking into account 
the recursive formula for $E^W_k$ for the
system $A_n$ given in \cite{Sawyer}, we conjecture that  the formulas \eqref{Amri1} and
\eqref{Amri2} generalize to recursive formulas for $E_k$ for the system $A_n$.
\end{remark}

\section{Sharp estimates of the Dunkl kernel and Dunkl heat kernel  on $A_2$}

For simplicity, we denote the Weyl chambers by $C_{ijk}= \{Y\colon~ y_i\ge y_j\ge y_k\}$. In particular $\overline{C^+}= C_{123}$.

\begin{theorem}\label{sharp}
We have the following estimates of the  Dunkl kernel for the system $A_2$.
\begin{itemize}
\item[(i)] For $X$ in the closed positive Weyl chamber $\overline{C^+}$ 
$$
E_k(X, \lambda)\asymp \frac{e^{\langle \lambda,X\rangle}}{
 (1+ \al\ax)^k(1+ \bl\bx)^k(1+ \gl\gx)^k}
$$
\item[(ii)] For $X$  in the Weyl chamber $\sigma_\alpha \overline{C^+} = C_{213}
 $
$$
E_k(X, \lambda)\asymp \frac{e^{\langle \lambda,X^+\rangle}}{
 (1+ \al\apx)^{k+1}(1+ \bl\bpx)^k(1+ \gl\gpx)^k}.
$$
For $X$  in the Weyl chamber $\sigma_\beta \overline{C^+} = C_{132}
 $
$$
E_k(X, \lambda)\asymp \frac{e^{\langle \lambda,X^+\rangle}}{
 (1+ \al\apx)^{k}(1+ \bl\bpx)^{k+1}(1+ \gl\gpx)^k}
$$
\item[(iii)] For $X$ in the Weyl chamber $
\sigma_\alpha\sigma_\beta \overline{C^+}
= \sigma_\beta \sigma_\gamma \overline{C^+}=\sigma_\gamma\sigma_\alpha\overline{C^+}=C_{231}$ 

$$
E_k(X, \lambda)
\asymp 
\frac{e^{\langle \lambda,X^+\rangle}}{
 (1+ \al\apx)^{k+1}(1+ \bl\bpx)^{k+1}(1+ \gl\gpx)^{k}} \quad {\rm if\  \al\le\bl}
$$

$$
E_k(X, \lambda)
\asymp \frac{e^{\langle \lambda,X^+\rangle}}{
 (1+ \al\apx)^{k}(1+ \bl\bpx)^{k+1}(1+ \gl\gpx)^{k+1}} \quad {\rm if\ }  \al\ge\bl,
 \al\apx\ge \bl\bpx
$$

$$
E_k(X, \lambda)
\asymp \frac{e^{\langle \lambda,X^+\rangle}}{
 (1+ \al\apx)^{k+1}(1+ \bl\bpx)^{k}(1+ \gl\gpx)^{k+1}} \quad {\rm if\ }  \al\ge\bl,
 \al\apx\le \bl\bpx.
$$

\item[(iv)] for
$X$ in the Weyl chamber $
\sigma_\beta \sigma_\alpha \overline{C^+}
= \sigma_\gamma   \sigma_\beta\overline{C^+}=\sigma_\alpha \sigma_\gamma\overline{C^+}
=C_{312}$

$$E_k(X, \lambda)
\asymp 
\frac{e^{\langle \lambda,X^+\rangle}}{
 (1+ \al\apx)^{k+1}(1+ \bl\bpx)^{k+1}(1+ \gl\gpx)^{k}} \quad {\rm if\ \bl\le\al}
$$

$$
E_k(X, \lambda)
\asymp \frac{e^{\langle \lambda,X^+\rangle}}{
 (1+ \al\apx)^{k}(1+ \bl\bpx)^{k+1}(1+ \gl\gpx)^{k+1}} \quad {\rm if\ }  \bl\ge\al,
 \al\apx\ge \bl\bpx
$$

$$
E_k(X, \lambda)
\asymp \frac{e^{\langle \lambda,X^+\rangle}}{
 (1+ \al\apx)^{k+1}(1+ \bl\bpx)^{k}(1+ \gl\gpx)^{k+1}} \quad {\rm if\ }  \bl\ge\al,
 \al\apx\le \bl\bpx.
$$

\item[(v)] For $X$  in the Weyl chamber $\sigma_\gamma \overline{C^+}=
\sigma_\alpha \sigma_\beta  \sigma_\alpha \overline{C^+} =\sigma_\beta \sigma_\alpha \sigma_\beta \overline{C^+} = C_{321},
 $
$$
E_k(X, \lambda)\asymp \frac{e^{\langle \lambda,X^+\rangle}}{
 (1+ \al\apx)^{k}(1+ \bl\bpx)^{k}(1+ \gl\gpx)^{k+1}}
$$
\end{itemize}
\end{theorem}

 Theorem \ref{sharp}
 and the estimates of $E_k(X,\lambda)$ obtained for the system $A_1$ in \cite{AnkerDT} support
 the next conjecture. 

 Each Weyl chamber $C$ is the image  of $C^+$ by an element $w$ of the Weyl group, i.e. $wC^+=C$.  Each element $w\in W$ decomposes (often non-uniquely)  as a composition of symetries
\begin{equation}\label{w}   
w=\sigma_{\alpha_1}\ldots \sigma_{\alpha_s} 
\end{equation}
We then call the sequence $(\sigma_{\alpha_1},...,\sigma_{\alpha_s})$ a realization of $C$.  Let ${\mathcal Short}(C)\subset W$  be the set of the shortest realizations of $C$, i.e. the number $s$ of symetries in the decomposition 
 \eqref{w} is minimised on elements of ${\mathcal Short}(C)\subset W$.
For example, for the system $A_2$, we have 
 ${\mathcal Short}(C231)=\{ (\sigma_\alpha,\sigma_\beta),  (\sigma_\beta,\sigma_\gamma),
 (\sigma_\gamma, \sigma_\alpha)\}$.
\begin{conjecture}
 For the root system
 $A_n$ the following estimate of the Dunkl kernel holds.

 Let $C$ a Weyl chamber. 
For all  $\lambda \in C^+$ and  $X\in C$ there exists $S\in Short(C)$ such that
$$E_k(X, \lambda)\asymp e^{\langle \lambda,X^+\rangle}
F(X,\lambda)
$$
and 
$$
F(X,\lambda)=
\frac{1}{
\prod_{\alpha>0}  (1+ \al\apx)^{p(\alpha)}},
$$
where $p(\alpha)=k+1 $ if
$\sigma_\alpha\in S$ and $p(\alpha)=k$ otherwise.

In particular, for $X\in C^+$,
$$F(X,\lambda)=
\frac{1}{\prod_{\alpha>0}  (1+ \al\ax)^{k}}
$$
and, for $X\in \sigma_{\beta} C^+$,
$$F(X,\lambda)=
\frac{1}{
 (1+ \bl\bx)^{k+1}
\prod_{\alpha\in\Sigma^+\setminus\{\beta\}}  (1+ \al\ax)^{k}}.
$$
\end{conjecture}

\begin{remark}
The kernel $E_k(X,\lambda)$ is analytic and therefore continuous.  This must be reflected in the estimates given in  Theorem \ref{sharp},  in particular in the rational portion of the estimate as the projection $X\mapsto X^+$ is continuous.

When $X$ traverses from $C^+$ to $\sigma_\alpha C^+$, $\alpha_{X^+}=0$.  This explains how the term $1+\al\alpha_{X^+}$ can  continuously change from exponent $k$ to $k+1$.  The same reasoning applies when $X$ traverses from $C^+$ to $\sigma_\beta C^+$.  

 When $X$ traverses from $\sigma_\alpha C^+$ (where $x_2>x_1>x_3$) to $\sigma_\beta \sigma_\alpha C^+$ (where $x_2>x_3>x_1$), $\beta_{X^+}=0$.  This explains how the term $1+\bl\beta_{X^+}$ can continuously change from exponent $k$ to $k+1$.  

 The same Weyl chamber 
 may be represented in two different ways, as  $\sigma_\gamma   \sigma_\beta\overline{C^+}=\sigma_\alpha \sigma_\gamma\overline{C^+}$ and
 the same reasoning applies 
 to explain the continuous change and appearance of powers $k+1$ of the two  other  terms.

 When $X$ traverses from $\sigma_\beta\sigma_\alpha C^+$ (where $x_2>x_3>x_1$) to $\sigma_\alpha\sigma_\beta\sigma_\alpha  C^+$ (where $x_3>x_2>x_1$), $\alpha_{X^+}=0$ and therefore $\gamma_{X^+}=\beta_{X^+}$.  This explains how the term $1+\al\alpha_{X^+}$ can continuously change from exponent $k+1$ to $k$ and the terms $1+\bl\beta_{X^+}$ and $1+\gl\gamma_{X^+}$ can 
 continuously ``exchange'' exponents.
 \end{remark} 

 \begin{corollary}\label{heat}
Let $X\in C^+$, $Y\in\a$.  We have
\begin{align*}
p_t(X,Y)\asymp \frac{t^{-4+k_\alpha+k_\beta+k_\gamma}\,e^{-|X-Y^+|^2}/(4t)}{
(t+\alpha_{X}\alpha_{Y^+})^{k_\alpha}
(t+\beta_{X}\beta_{Y^+})^{k_\beta}
(t+\gamma_{X}\gamma_{Y^+})^{k_\gamma}
}
\end{align*}
where $w\in W$ is such that $Y=w Y^+$ and
\begin{align*}
(k_\alpha,k_\beta,k_\gamma)=\left\lbrace
\begin{array}{cl}
(k,k,k)&\text{if $w=id$}\\
(k+1,k,k)&\text{if $w=\sigma_\alpha$}\\
(k,k+1,k)&\text{if $w=\sigma_\beta$}\\
(k,k,k+1)&\text{if $w=\sigma_\gamma$}\\
(k+1,k+1,k)
\ {\rm or}\   (k+1,k,k+1 )
\ {\rm or}  \  (k,k+1,k+1 )
&\text{if $w=\sigma_\alpha\sigma_\beta$ or $w=\sigma_\beta
\sigma_\alpha$}
\end{array}
\right.
\end{align*}
When $w=\sigma_\alpha\sigma_\beta$ or $w=\sigma_\beta
\sigma_\alpha$,
the conditions for the appearance of one of three possible values of
$(k_\alpha,k_\beta,k_\gamma)$ are given in Theorem \ref{sharp}.
\end{corollary}

\begin{proof}
The Dunkl heat kernel $p_t(X,Y)$ is given as
\begin{align} 
p_t(X,Y)
&=\frac1{2^{\gamma+d/2} c_k}\,t^{-\frac{d}{2}-\gamma}\,e^{\frac{-|X|^2-|Y|^2}{4t}}\,E_k\left(X,\frac{Y}{2t}\right),\label{heatR}
\end{align}
where $\gamma= \sum_{\alpha>0} k(\alpha)$.  It suffices then to use our estimates for $E_k(X,\lambda)$ (\cite[Lemma 4.5]{Roesler}).

\end{proof}

\section{Proof of the sharp estimate}
In this Section we 
give the proof of Theorem \ref{sharp}. 

\subsection{Substitution of the estimates of $E_k^{rk\,1}$ in the integral formulas for  $E_k$}

In the beginning of the proof, in the formulas from Theorem \ref{A2A1}, we replace the Dunkl kernels on $A_1$
by their estimates provided in \cite{AnkerDT}:
\begin{equation}\label{est:A1}
E_k^{\text{rk\,1}}( x,y) \asymp \frac{e^{|xy|}}{(1+|xy|)^{k+p}} \quad 
\ \text{where $p=0$ if $xy\ge 0$ and $p=1$ if $xy\le 0$.}
 \end{equation}
Let $J(X,\lambda)=V(\lambda)^{2k}\,E_k(X,\lambda)$.
Estimating $E_k(X,\lambda)$ is equivalent to estimating of $J(X,\lambda)$.

\begin{lemma} \label{J1alpha}

If $x_1\geq x_2$, we have
\begin{align}
J(X,\lambda)
&\asymp\int_{\lambda_3}^{\lambda_2} \int_{\lambda_2}^{\lambda_1} 
\frac{\left(\nu_{1}-\nu_{2}\right) (\al+\alpha_{X}\left(\nu_{1}-\lambda_{2}\right) \left(\lambda_{1}-\nu_{2}\right))}{(1+\alpha_X(\nu_1-\nu_2))^{k+1}}
\nonumber
\\
\label{alpha+}
 &\qquad\qquad\qquad
\,(\nu_1-\lambda_3) (\nu_2-\lambda_3)
e^{(x_{1}-x_{3}) \nu_{1} +(x_{2}-x_{3}) \nu_{2}}
\,W_k(\lambda,\nu),d\nu_1d\nu_2.
\end{align}

 If $x_2\geq x_1$, we have
\begin{align}
J(X,\lambda)&\asymp
\int_{\lambda_3}^{\lambda_2} \int_{\lambda_2}^{\lambda_1} 
\frac{\left(\nu_{1}-\nu_{2}\right) \left(\alpha_{\lambda}-\alpha_{X} \left(\lambda_{1}-\nu_{1}\right) \left(\lambda_{2}-\nu_{2}\right)\right)}{(1-\alpha_X(\nu_1-\nu_2))^{k+1}}
\nonumber
\\
\label{alpha-}
&\qquad\qquad\qquad
\,(\nu_1-\lambda_3) (\nu_2-\lambda_3)
e^{(x_{2}- x_{3}) \nu_{1}+( x_{1}-x_{3}) \nu_{2}}
\,W_k(\lambda,\nu),d\nu_1d\nu_2.
\end{align}

\end{lemma}

\begin{proof}

If $x_2\geq x_1$, we have
\begin{align*}
J(X,\lambda)&\asymp\int_{\lambda_3}^{\lambda_2} \int_{\lambda_2}^{\lambda_1} 
\{
(\nu_1-\lambda_2)(\lambda_1-\nu_2)\,E_k^{\text{rk\,1}}((x_1-x_2)/2,(\nu_1-\nu_2))
\\ &\qquad\qquad\qquad
+(\lambda_1-\nu_1)(\lambda_2-\nu_2)\,E_k^{\text{rk\,1}}(-(x_1-x_2)/2,(\nu_1-\nu_2))
\}
\\ &\qquad\qquad\qquad
\,(\nu_1-\lambda_3) (\nu_2-\lambda_3)
e^{(x_1+x_2-2x_3)(\nu_1+\nu_2)/2}
\,W_k(\lambda,\nu),d\nu_1d\nu_2.\\
&\asymp\int_{\lambda_3}^{\lambda_2} \int_{\lambda_2}^{\lambda_1} 
\{
(\nu_1-\lambda_2)(\lambda_1-\nu_2)\,\frac{e^{(x_2-x_1)(\nu_1-\nu_2)/2}}{(1-\alpha_X(\nu_1-\nu_2))^{k+1}}
\\ &\qquad\qquad\qquad
+(\lambda_1-\nu_1)(\lambda_2-\nu_2)\,\frac{e^{(x_2-x_1)(\nu_1-\nu_2)/2}}{(1-\alpha_X(\nu_1-\nu_2))^{k}}
\}
\\ &\qquad\qquad\qquad
\,(\nu_1-\lambda_3) (\nu_2-\lambda_3)
e^{(x_1+x_2-2x_3)(\nu_1+\nu_2)/2}
\,W_k(\lambda,\nu),d\nu_1d\nu_2\\
&=\int_{\lambda_3}^{\lambda_2} \int_{\lambda_2}^{\lambda_1} 
\frac{(\nu_1-\lambda_2)(\lambda_1-\nu_2)
+(\lambda_1-\nu_1)(\lambda_2-\nu_2)\,(1-\alpha_X(\nu_1-\nu_2))}{(1-\alpha_X(\nu_1-\nu_2))^{k+1}}
\\ &\qquad\qquad\qquad
\,(\nu_1-\lambda_3) (\nu_2-\lambda_3)
e^{x_{2} \nu_{1}-\nu_{1} x_{3}+\nu_{2} x_{1}-x_{3} \nu_{2}}
\,W_k(\lambda,\nu),d\nu_1d\nu_2\\
&=\int_{\lambda_3}^{\lambda_2} \int_{\lambda_2}^{\lambda_1} 
\frac{\left(\nu_{1}-\nu_{2}\right) \left(\alpha_{\lambda}-\alpha_{X} \left(\lambda_{1}-\nu_{1}\right) \left(\lambda_{2}-\nu_{2}\right)\right)}{(1-\alpha_X(\nu_1-\nu_2))^{k+1}}
\\ &\qquad\qquad\qquad
\,(\nu_1-\lambda_3) (\nu_2-\lambda_3)
e^{x_{2} \nu_{1}-\nu_{1} x_{3}+\nu_{2} x_{1}-x_{3} \nu_{2}}
\,W_k(\lambda,\nu),d\nu_1d\nu_2.
\end{align*}

If $x_1\geq x_2$, we have
\begin{align*}
J(X,\lambda)
&\asymp\int_{\lambda_3}^{\lambda_2} \int_{\lambda_2}^{\lambda_1} 
\frac{\left(\nu_{1}-\nu_{2}\right) (\lambda_{1}-\lambda_{2}+\alpha_{X}\left(\nu_{1}-\lambda_{2}\right) \left(\lambda_{1}-\nu_{2}\right))}{(1+\alpha_X(\nu_1-\nu_2))^{k+1}}
\\ &\qquad\qquad\qquad
\,(\nu_1-\lambda_3) (\nu_2-\lambda_3)
e^{x_{1} \nu_{1}-\nu_{1} x_{3}+x_{2} \nu_{2}-x_{3} \nu_{2}}
\,W_k(\lambda,\nu),d\nu_1d\nu_2.
\end{align*}
    
\end{proof}

Interchanging the roles of the roots $\alpha$ and $\beta$, we find

\begin{lemma}\label{J1beta}
    
If $x_2\geq x_3$, we have
\begin{align}
J(X,\lambda)
&\asymp \int_{\lambda_3}^{\lambda_2} \int_{\lambda_2}^{\lambda_1} 
\frac{\left(\nu_{1}-\nu_{2}\right) \left(\bl+\beta_X \left(\nu_{1}-\lambda_{3}\right) \left(\lambda_{2}-\nu_{2}\right)\right)}{(1+\beta_X(\nu_1-\nu_2))^{k+1}}
\nonumber
\\
\label{beta+}
 &\qquad\qquad\qquad
\,(\lambda_1-\nu_1) (\lambda_1-\nu_2)
e^{\nu_1(x_2-x_1)+\nu_2(x_3-x_1)}
\,W_k(\lambda,\nu),d\nu_1d\nu_2.
\end{align}

If $x_3\geq x_2$, we have
\begin{align}
J(X,\lambda)
&\asymp \int_{\lambda_3}^{\lambda_2} \int_{\lambda_2}^{\lambda_1} 
\frac{\left(\nu_{1}-\nu_{2}\right) \left(\bl-\beta_X \left(\nu_{1}-\lambda_{2}\right) \left(\nu_{2}-\lambda_{3}\right)\right)}{(1-\beta_X(\nu_1-\nu_2))^{k+1}}
\nonumber
\\
\label{beta-}
 &\qquad\qquad\qquad
\,(\lambda_1-\nu_1) (\lambda_1-\nu_2)
e^{\nu_1(x_3-x_1)+\nu_2(x_2-x_1)}
\,W_k(\lambda,\nu),d\nu_1d\nu_2.
\end{align}

\end{lemma}

\subsection{Reduction to estimating  $I(X,\lambda)=    e^{-\langle \lambda, X^+\rangle}
J(X,\lambda)$ \label{sec:I}
}

Now, in the formulas from Lemmas \ref{J1alpha} and \ref{J1beta}, on each Weyl chamber, we will 
put in evidence the exponential factor $e^{\langle \lambda, X^+\rangle}$.
Estimating $E_k(X,\lambda)$ is equivalent to
 estimating the integral $I(X,\lambda)$.

\subsubsection{Formulas \eqref{alpha+} and  \eqref{alpha-}} Let
$Q_\alpha^+=e^{\nu_1(x_1-x_3)+\nu_2(x_2-x_3)}$  and $Q_\alpha^-=e^{\nu_{1}(x_2-x_3)+\nu_2(x_1-x_3)}$.
We have 
\begin{enumerate}

\item For $X \in  \overline{C^+} = C_{123}$: $Q_\alpha^+=e^{\langle\lambda, X^+\rangle}  \,e^{-\left(x_{1}-x_{3}\right) \left(\lambda_{1}-\nu_{1}\right)-\left(x_{2}-x_{3}\right) \left(\lambda_{2}-\nu_{2}\right)}$.

\item For $X \in \sigma_\gamma \overline{C^+} = C_{321}$: $Q_\alpha^-=e^{\langle\lambda, X^+\rangle}  \,e^{-\left(x_{3}-x_{2}\right) \left(\nu_{1}-\lambda_{2}\right)-\left(x_{3}-x_{1}\right) \left(\nu_{2}-\lambda_{3}\right)}$.
\item For $X \in \sigma_\beta\sigma_\alpha \overline{C^+} = C_{231}$: $Q_\alpha^-=e^{\langle\lambda, X^+\rangle} \,e^{-\left(x_{2}-x_{3}\right) \left(\lambda_{1}-\nu_{1}\right)-\left(x_{3}-x_{1}\right) \left(\nu_{2}-\lambda_{3}\right)}$.

\item For $X \in \sigma_\alpha \sigma_\beta\overline{C^+} = C_{312}$: $Q_\alpha^+=e^{\langle\lambda, X^+\rangle}  \,e^{-\left(x_{3}-x_{1}\right) \left(\nu_{1}-\lambda_{2}\right)-\left(x_{3}-x_{2}\right) \left(\nu_{2}-\lambda_{3}\right)}
$.

\item
For $X \in \sigma_\alpha \overline{C^+} = C_{213}$:
$Q_\alpha^-=e^{\langle\lambda, X^+\rangle} \,e^{-\left(x_{2}-x_{3}\right) \left(\lambda_{1}-\nu_{1}\right)-\left(x_{1}-x_{3}\right) \left(\lambda_{2}-\nu_{2}\right)}
$.

\item For $X \in \sigma_\beta \overline{C^+} = C_{132}$: $Q_\alpha^+=e^{\langle\lambda, X^+\rangle}  \,e^{-\left(x_{1}-x_{3}\right) \left(\lambda_{1}-\nu_{1}\right)-\left(x_{3}-x_{2}\right) \left(\nu_{2}-\lambda_{3}\right)}
$.
\end{enumerate}

\subsubsection{Formulas \eqref{beta+} and  \eqref{beta-}} Let  $Q_\beta^+=e^{\nu_1(x_2-x_1)+\nu_2(x_3-x_1)}$
and  $Q_\beta^-=e^{\nu_1(x_3-x_1)+\nu_2(x_2-x_1)}$. We have 
\begin{enumerate}
\item For $X \in\overline{C^+} = C_{123}$: $Q_\beta^+=e^{\langle\lambda, X^+\rangle}\,e^{-\left(x_{1}-x_{2}\right) \left(\nu_{1}-\lambda_{2}\right)-\left(x_{1}-x_{3}\right) \left(\nu_{2}-\lambda_{3}\right)}
$.
\item For $X \in \sigma_\gamma \overline{C^+} = C_{321}$: $Q_\beta^-=e^{\langle\lambda, X^+\rangle} \,e^{-\left(x_{3}-x_{1}\right) \left(\lambda_{1}-\nu_{1}\right)-\left(x_{2}-x_{1}\right) \left(\lambda_{2}-\nu_{2}\right)}
$.

\item For $X \in \sigma_\beta\sigma_\alpha overline{C^+} = C_{231}$: $Q_\beta^+=e^{\langle\lambda, X^+\rangle} \,e^{-\left(x_{2}-x_{1}\right) \left(\lambda_{1}-\nu_{1}\right)-\left(x_{3}-x_{1}\right) \left(\lambda_{2}-\nu_{2}\right)}
$.

\item For $X \in \sigma_\alpha \sigma_\beta\overline{C^+} = C_{312}$: $Q_\beta^-=e^{\langle\lambda, X^+\rangle} \,e^{-\left(x_{3}-x_{1}\right) \left(\lambda_{1}-\nu_{1}\right)-\left(x_{1}-x_{2}\right) \left(\nu_{2}-\lambda_{3}\right)}
$.
\item For $X \in \sigma_\alpha \overline{C^+} = C_{213}$: $Q_\beta^+=e^{\langle\lambda, X^+\rangle} \,e^{-\left(x_{2}-x_{1}\right) \left(\lambda_{1}-\nu_{1}\right)-\left(x_{1}-x_{3}\right) \left(\nu_{2}-\lambda_{3}\right)}$.

\item For $X \in \sigma_\beta \overline{C^+} = C_{132}$: $Q_\beta^-=e^{\langle\lambda, X^+\rangle}\,e^{-\left(x_{1}-x_{3}\right) \left(\nu_{1}-\lambda_{2}\right)-\left(x_{1}-x_{2}\right) \left(\nu_{2}-\lambda_{3}\right)}
$.
\end{enumerate}

\subsection{Reductions using symmetries between the Weyl chambers}

Consider the following relations: 
\begin{align*}
E((x_1,x_2,x_3),(\lambda_1,\lambda_2,\lambda_3))
&=E((-x_1,-x_2,-x_3),(-\lambda_1,-\lambda_2,-\lambda_3))\\
&=E((-x_3,-x_2,-x_1),(-\lambda_3,-\lambda_2,-\lambda_1))
\end{align*}
with $\tilde{\lambda}= (-\lambda_3,-\lambda_2,-\lambda_1) \in\a^+$ since $\lambda\in\a^+$.  Write $\tilde{X}=(-x_3,-x_2,-x_1)$. 
This ``symmetry'' sends $X\in C_{123}$ to  $\tilde{X}\in C_{123}$, $X\in C_{213}$ to  $\tilde{X}\in C_{132}$, $X\in C_{231}$ to  $\tilde{X}\in C_{312}$ and $X\in C_{321}$ to  $\tilde{X}\in C_{321}$.  Note also that $\alpha_{\tilde{\lambda}}=\beta_\lambda$ and  $\beta_{\tilde{\lambda}}=\alpha_\lambda$.  Thus, we reduce finding the estimates of $E_k$ to the four Weyl chambers: $C_{123}$, $C_{213}$, $C_{231}$ and $C_{321}$.

There are several other symmetries (such as exchanging the role between $\alpha$ and $\beta$ or 
commuting $X$ and $\lambda$) but they do not give rise to any other reduction.





\subsection{ Plan of the proof for a fixed Weyl chamber}

Fix $X$ in a Weyl chamber.  The proof will be done separately for $\al\le \bl$ and for $\bl\le\al$.

{\bf Step 1.}
Choose a ``starting'' formula for $I(X,\lambda)$ either with $\alpha_X$ (i.e.  using $J$ expressed by the formula 
 \eqref{alpha+} or \eqref{alpha-})  or with $\beta_X$ (i.e.   using $J$ expressed by the formula  formula 
  \eqref{beta+} or  \eqref{beta-}).   A choice allowing the next Step 2 always exists. This depends on the form of the exponentials $Q_\alpha^\pm$ or $Q_\beta^\pm$ which is crucial in this choice. 

 {\bf Step 2.}   In the double integral for $I(X,\lambda)$ chosen in Step 1, choose a half-subintegral $I_1$ such that
 it may be proven that $I_1\gtrsim I_2=I-I_1$.

 {\bf Step 3.}  Find a sharp estimate of $I_1$. It will be a sharp estimate of $I$.

\subsection{Positive Weyl chamber}

Suppose $\al\geq\bl$. We choose as  the starting  formula the formula \eqref{alpha+} with $\alpha_X$.
Using the corresponding formula for $Q_\alpha^+$  of Section \ref{sec:I},  we have
\begin{align*}
I(X,\lambda) 
&\asymp\int_{\lambda_3}^{\lambda_2} \int_{\lambda_2}^{\lambda_1} 
\frac{\left(\nu_{1}-\nu_{2}\right) (\al+\alpha_{X}\left(\nu_{1}-\lambda_{2}\right) \left(\lambda_{1}-\nu_{2}\right))}{(1+\alpha_X(\nu_1-\nu_2))^{k+1}}
\\
 &\qquad\qquad\qquad
\,(\nu_1-\lambda_3) (\nu_2-\lambda_3)
e^{-\left(x_{1}-x_{3}\right) \left(\lambda_{1}-\nu_{1}\right)-\left(x_{2}-x_{3}\right) \left(\lambda_{2}-\nu_{2}\right)}
\,W_k(\lambda,\nu),d\nu_1d\nu_2.
\end{align*}

We set $I_1=\int_{\lambda_3}^{\lambda_2}\,\int_{M_1}^{\lambda_1}[\dots]$ and $I_2=I-I_1$. Denote by $\tilde I_i$, $i=1,2$, the integral in the variable $\nu_1$. We have
\begin{align*}
\tilde{I}_1\geq e^{-\alpha_\lambda \gx/3}
\int_{(2\,\lambda_1+\lambda_2)/3}^{(3\,\lambda_1+\lambda_2)/4} \ldots d\nu_1
\gtrsim   \frac{\al^{1+1+k+2(k-1)+1}}{(1+\alpha_X\al)^{k+1}}\,(1+\alpha_X\al)e^{-\alpha_\lambda\gx/3}=\frac{\al^{3k+1}}{(1+\alpha_X\al)^{k}}e^{-\alpha_\lambda\gx/3}
\end{align*}
while, using  $\lambda_1-\nu_2\lesssim \al $
and $\nu_1-\lambda_2\le \nu_1-\nu_2 $
\begin{align*}
\tilde{I}_2&\lesssim \al^{k-1}e^{-\alpha_\lambda\gx/2} \int_{\lambda_2}^{M_1}\,(\nu_1-\nu_2)\frac{\al+\alpha_X(\nu_1-\lambda_2)(\lambda_1-\nu_2)}{(1+\alpha_X(\nu_1-\nu_2))^{k+1}}\,(\nu_1-\lambda_2)^{k-1}\,(\nu_1-\lambda_3)^k\,d\nu_1\\
&\leq\al^{k-1+1+k} e^{-\alpha_\lambda\gx/2} \int_{\lambda_2}^{M_1}\,\frac{(\nu_1-\nu_2)(\nu_1-\lambda_2)^{k-1}}{(1+\alpha_X(\nu_1-\nu_2))^{k}}\,d\nu_1.
\end{align*}

Now, if $k\leq1$, since $\nu_1\mapsto \frac{\nu_1-\nu_2}{(1+\alpha_X(\nu_1-\nu_2))^{k}}$ is increasing, we have
\begin{align} \label{est:I2less1}
\int_{\lambda_2}^{M_1}\,\frac{(\nu_1-\nu_2)(\nu_1-\lambda_2)^{k-1}}{(1+\alpha_X(\nu_1-\nu_2))^{k}}\,d\nu_1
\lesssim \frac{\al}{(1+\alpha_X\al)^{k}}
\int_{\lambda_2}^{M_1}\,(\nu_1-\lambda_2)^{k-1}d\nu_1 \asymp \frac{\al^{k+1}}{(1+\alpha_X\al)^{k}}.
\end{align}

If $k\geq1$, since 
$(\nu_1-\lambda_2)^{k-1}\le
(\nu_1-\nu_2)^{k-1}$ and
$\nu_1\mapsto \frac{\nu_1-\nu_2}{1+\alpha_X(\nu_1-\nu_2)}$ is increasing, we have
\begin{align}\label{est:I2more1}
\int_{\lambda_2}^{M_1}\,\frac{(\nu_1-\nu_2)(\nu_1-\lambda_2)^{k-1}}{(1+\alpha_X(\nu_1-\nu_2))^{k}}\,d\nu_1
\leq \frac{\al^k}{(1+\alpha_X\al)^{k}}
\int_{\lambda_2}^{M_1}\,d\nu_1 \asymp \frac{\al^{k+1}}{(1+\alpha_X\al)^{k}}.
\end{align}

We will use the same device (separating $k\leq 1$ and $k\geq 1$) in \eqref{K2},  \eqref{K3} and in \eqref{K1} without repeating the details.  In both cases, $\tilde{I}_2\lesssim \tilde{I}_1$.

Finally, 
\begin{align*}
I_1 &\asymp \frac{\al^{1+1+k+2(k-1)}}{(1+\alpha_X\al)^{k}}
\int_{M_1}^{\lambda_1} e^{-\gx\,(\lambda_1-\nu_1)}\,(\lambda_1-\nu_1)^{k-1}\,d\nu_1
\cdot
\int_{\lambda_3}^{\lambda_2} e^{-{ \bx}\,(\lambda_2-\nu_2)}\,(\lambda_2-\nu_2)^{k-1}(\nu_2-\lambda_3)^{k-1}\,d\nu_1\\
&\asymp \frac{\al^{3k}}{(1+\alpha_X\al)^{k}} \frac{\al^{k}}{(1+\gx\al)^{k}}  \frac{\bl^{2k}}{(1+\bx\bl)^{k}}
\end{align*}

In  the first integral we make the change of variable $u=\gx(\lambda_1-\nu_1)$  and we use
\begin{align}\label{est:rank0}
    \int_0^x e^{-u} u^{k-1}du \asymp x^k(1+x)^{-k}.
\end{align}
We obtain 
$$
\int_{M_1}^{\lambda_1} e^{-\gx\,(\lambda_1-\nu_1)}\,(\lambda_1-\nu_1)^{k-1}\,d\nu_1\asymp
\frac{\al^k}{(1+\gx\al)^k}.$$

In the second integral, we change $v=\lambda_2-\nu_2$ and next $v=\bl t$. We obtain
$$
\bl^{2k} \int_0^1 e^{-\bx\bl t} 
t^{k-1} (1-t)^k dt$$
We use the integral representation \eqref{rank1int} of the function $E^{\rm rk\,1}(\bx/2, \bl)$ and we apply the estimate \eqref{est:A1} which is the desired result since $\al\asymp\al+\bl=
 \gl$.

Note that if we assume $\bl\geq \al$, a similar proof holds using  the starting  formula
the formula \eqref{beta+} with $\beta_X$ and we set
$I_1=\int_{\lambda_3}^{M_2}\,\int_{\lambda_2}^{\lambda_1}[\dots]$.

\subsection{Weyl chamber $C_{321}$}

Suppose $\bl\geq\al$.
We choose as  the starting  formula
the formula \eqref{alpha-} with $\alpha_X$.
By Section \ref{sec:I}, 
the exponent
$Q_\alpha^-=e^{\langle\lambda, X^+\rangle}  \,e^{\bx \left(\nu_{1}-\lambda_{2}\right)+\gx \left(\nu_{2}-\lambda_{3}\right)}$ 
and the starting formula
is
\begin{align*}
I(X,\lambda)&\asymp
\int_{\lambda_3}^{\lambda_2} \int_{\lambda_2}^{\lambda_1} 
\frac{\left(\nu_{1}-\nu_{2}\right) \left(\alpha_{\lambda}-\alpha_{X} \left(\lambda_{1}-\nu_{1}\right) \left(\lambda_{2}-\nu_{2}\right)\right)}{(1-\alpha_X(\nu_1-\nu_2))^{k+1}}
\nonumber
\\
&\qquad\qquad\qquad
\,(\nu_1-\lambda_3) (\nu_2-\lambda_3)
e^{\bx \left(\nu_{1}-\lambda_{2}\right)+\gx \left(\nu_{2}-\lambda_{3}\right)}
\,W_k(\lambda,\nu),d\nu_1d\nu_2.
\end{align*}

We set $I_1=\int_{\lambda_3}^{M_2}\,\int_{\lambda_2}^{\lambda_1}[\dots]$ and $I_2=I-I_1$. 
Denote by $\tilde I_i$, $i=1,2$, the integrals in the variable $\nu_2$.

We have,
\begin{align*}
\tilde{I}_1&\geq \int_{(3\,\lambda_3+\lambda_2)/4}^{(2\,\lambda_3+\lambda_2)/3}
\ldots d\nu_2
\gtrsim  e^{\gx\bl/3} \frac{\bl^{1+k+2(k-1)+1}}{(1-\alpha_X\bl)^{k+1}}\,(\al-\alpha_X\bl(\lambda_1-\nu_1))\\
&\gtrsim  e^{-(x_3-x_1)\bl/3} \frac{\bl^{3k}}{(1-\alpha_X\bl)^{k+1}}\,(\al-\alpha_X\bl(\lambda_1-\nu_1))
\end{align*}
while
\begin{align}
\tilde{I}_2&\lesssim e^{-(x_3-x_1)\bl/2} \bl^{k} \,\int_{M_2}^{\lambda_2}\,\frac{(\nu_1-\nu_2)(\al-\alpha_X(\lambda_1-\nu_1)(\lambda_2-\nu_2))
(\lambda_1-\nu_2)^{k-1}(\lambda_2-\nu_2)^{k-1}}{(1-\alpha_X(\nu_1-\nu_2))^{k+1}}\,d\nu_2\nonumber\\
&\leq e^{-(x_3-x_1)\bl/2} \al \bl^{k} \,\int_{M_2}^{\lambda_2}\,\frac{(\nu_1-\nu_2)
(\lambda_1-\nu_2)^{k-1}(\lambda_2-\nu_2)^{k-1}}{(1-\alpha_X(\nu_1-\nu_2))^{k}}\,d\nu_2
\lesssim \frac{e^{-(x_3-x_1)\bl/2} \al \bl^{3k}}{(1-\alpha_X\bl)^{k}}.\label{K2}
\end{align}

We now proceed similarly as in \eqref{est:I2less1}
and \eqref{est:I2more1}.
The case $k\ge 1$ is similar to $C^+$.
For $k\le 1$, we use 
 $$
 \int_{M_2}^{\lambda_2}
 (\lambda_1-\nu_2)^{k-1}
 (\lambda_2-\nu_2)^{k-1}d\nu_2
= \int_{0}^{\al/2} u^{k-1}(\al-u)^{k-1}\,du
=\al^{2k-1} \int_{0}^{1/2} v^{k-1}(1-v)^{k-1}
\,dv.
 $$

Now, since $x_3-x_1\geq -\alpha_X=x_2-x_1$, we have
\begin{align*}
\frac{\tilde{I}_1}{\tilde{I}_2}&\gtrsim  \frac{e^{(x_3-x_1)\bl/6}\,(\al-\alpha_X(\lambda_1-\nu_1))}{\al(1-\alpha_X\bl)}
\gtrsim  \frac{1+(x_3-x_1)\bl/6}{1-\alpha_X\bl}\geq 1/6.
\end{align*}

Now, estimating $I$ boils down to estimating the following three integrals:

\begin{align*}
I\asymp I_1&\asymp\frac{\bl^{3k-1}}{(1-\alpha_X \bl)^{k+1}}
\int_{\lambda_3}^{M_2}\,\int_{\lambda_2}^{\lambda_1}\,e^{-(x_3-x_2)(\nu_1-\lambda_2)-(x_3-x_1)(\nu_2-\lambda_3)}
(\al-\alpha_X(\lambda_1-\nu_1)\bl)
\\&\qquad\qquad
(\nu_2-\lambda_3)^k(\nu_1-\lambda_2)^{k-1}(\lambda_1-\nu_1)^{k-1}d\nu_1 d\nu_2\\
\\
&=\frac{\bl^{3k-1}}{(1-\alpha_X \bl)^{k+1}}
\left[ \int_{\lambda_3}^{M_2}\,e^{-(x_3-x_1)(\nu_2-\lambda_3)} (\nu_2-\lambda_3)^k d\nu_2\right]
\\&\qquad\qquad
\left[\al
\int_{\lambda_2}^{\lambda_1}\,e^{-(x_3-x_2)(\nu_1-\lambda_2)}
(\nu_1-\lambda_2)^{k-1}(\lambda_1-\nu_1)^{k-1}d\nu_1
\right.\\&\qquad\qquad\qquad\left.
-\alpha_X \bl \int_{\lambda_2}^{\lambda_1}\,e^{-(x_3-x_2)(\nu_1-\lambda_2)}
(\nu_1-\lambda_2)^{k-1}(\lambda_1-\nu_1)^{k}d\nu_1
\right]\\
\end{align*}

We estimate the first integral using the estimate \eqref{est:rank0}.
In the second integral, up to a natural change of variables, we recognize the integral formula for the modified Bessel function $\mathcal{J}_{k-1/2}$. Recall that $\mathcal{J}_{k-1/2}(vx) = E_k^{W, rk\,1} (v,x)$.
By the estimates of the function $\mathcal{J}_{k}$  (we can also use of  the estimates of the Dunkl kernel
in the $W$-invariant rank 1 case   \cite{PGPS0}), we have
 \begin{align}\label{Wrk1}
 \int_{\lambda_2}^{\lambda_1}e^{-(x_3-x_2)(\nu_1-\lambda_2)}
 (\lambda_1-\nu_1)^{k-1}(\nu_1-\lambda_2)^{k-1}
 d\nu_1
 \asymp \frac{\al^{2k-1}}{(1+(x_3-x_2)\al)^k}.
\end{align}
The third integral, similarly as in $C^+$, is estimated by the formula \eqref{est:A1} in the general rank 1 case. Finally, we obtain
\begin{align*}
I&\asymp\frac{\bl^{3k-1}}{(1-\alpha_X \bl)^{k+1}}
\frac{\bl^{k+1}}{(1+(x_3-x_1)\bl)^{k+1}}
\left[ \frac{\al \al^{2k-1}}{(1+(x_3-x_2)\al)^{k}}
-\frac{\alpha_X\bl \al^{2k}}{(1+(x_3-x_2)\al)^{k}}
\right]\\
&=\frac{\al^{2k}\bl^{4k}(1-\ax\bl)}{(1-\alpha_X \bl)^{k+1}(1+(x_3-x_2)\al)^{k}(1+(x_3-x_1)\bl)^{k+1}}.
\end{align*}

Note that if we assume $\al\geq \bl$, a similar proof holds using the {formula \eqref{beta-}} with
$I_1=\int_{\lambda_3}^{\lambda_2}
\int_{M_1}^{\lambda_1}$ and $I_2=I-I_1$.

\subsection{Weyl chamber $\sigma_\beta \sigma_\alpha \overline{C^+}$ (i.e.
$x_2\ge x_3 \ge x_1$)}

In the case $\al\ge \bl$, we use the {formula \eqref{beta+}} from Lemma \ref{J1beta}.

Discussions with Anker and Trojan were helpful in determining the right bounds in this chamber in the chamber $C_{312}$.

Let
\begin{align*}
I_1(X,\lambda)
&\asymp \int_{\lambda_3}^{\lambda_2} \int_{M_1}^{\lambda_1} e^{-\left(x_{2}-x_{1}\right) \left(\lambda_{1}-\nu_{1}\right)-\left(x_{3}-x_{1}\right) \left(\lambda_{2}-\nu_{2}\right)}
\frac{\left(\nu_{1}-\nu_{2}\right) \left(\bl+\beta_X \left(\nu_{1}-\lambda_{3}\right) \left(\lambda_{2}-\nu_{2}\right)\right)}{(1+\beta_X(\nu_1-\nu_2))^{k+1}}
\\
 &\qquad\qquad\qquad
\,(\lambda_1-\nu_1) (\lambda_1-\nu_2)
\,W_k(\lambda,\nu)\,d\nu_1d\nu_2
\end{align*}
and $I_2=I-I_1$.  We have
\begin{align*}
\tilde{I_1}&\geq \int_{(2\,\lambda_1+\lambda_2)/3}^{(3\,\lambda_1+\lambda_2)/4}\dots
\gtrsim \frac{\al^{1+k+2(k-1)+1} e^{-\left(x_{2}-x_{1}\right)/3}\,\left(\bl+\beta_X \al\left(\lambda_{2}-\nu_{2}\right)\right)}{(1+\beta_X\al)^{k+1}}.
\end{align*}

Next, for $\tilde{I_2}$,  we proceed similarly as in \eqref{est:I2less1} and \eqref{est:I2more1}.
\begin{align}
\tilde{I_2}&\lesssim \al^{k} \,e^{-\left(x_{2}-x_{1}\right)\al/2}\,\int_{\lambda_2}^{M_1}e^{-\left(x_{2}-x_{1}\right) \left(\lambda_{1}-\nu_{1}\right)\bl}
\frac{\left(\nu_{1}-\nu_{2}\right) (\nu_1-\lambda_2)^{k-1} (\nu_1-\lambda_3)^{k-1}  \left(\bl+\beta_X \left(\nu_1-\lambda_3\right)\bl\right)}{(1+\beta_X(\nu_1-\nu_2))^{k+1}} d\nu_1 \nonumber\\
&\leq \al^{k} \bl\,e^{-\left(x_{2}-x_{1}\right)\al/2}\int_{\lambda_2}^{M_1} 
\frac{\left(\nu_{1}-\nu_{2}\right) (\nu_1-\lambda_2)^{k-1} (\nu_1-\lambda_3)^{k-1} }{(1+\beta_X(\nu_1-\nu_2))^{k}} d\nu_1 
\lesssim \frac{\al^{3k} \bl\,e^{-\left(x_{2}-x_{1}\right) \al/2}
}{(1+\beta_X\al)^{k}} \label{K3}.
\end{align}

Hence,
\begin{align*}
\frac{\tilde{I_1}}{\tilde{I_2}}\gtrsim\frac{e^{\left(x_{2}-x_{1}\right) \al/6}\,\left(\bl+\beta_X \al\left(\lambda_{2}-\nu_{2}\right)\right)}{\bl(1+\beta_X\al)}
\geq \frac{1+\left(x_{2}-x_{1}\right) \al/6}{1+\beta_X\al}\geq 1/6.
\end{align*}

Now,
\begin{align*}
I_1(X,\lambda)
&\asymp \frac{\al^{1+k+2(k-1)}}{(1+\beta_X\al)^{k+1}} \int_{\lambda_3}^{\lambda_2} \int_{M_1}^{\lambda_1} e^{-\left(x_{2}-x_{1}\right) \left(\lambda_{1}-\nu_{1}\right)-\left(x_{3}-x_{1}\right) \left(\lambda_{2}-\nu_{2}\right)}
\left(\bl+\beta_X \al \left(\lambda_{2}-\nu_{2}\right)\right)
\\
 &\qquad\qquad\qquad
\,(\lambda_1-\nu_1)^k (\lambda_2-\nu_2)^{k-1}(\nu_2-\lambda_3)^{k-1} d\nu_1d\nu_2\\
&\asymp \frac{\al^{3k-1}}{(1+\beta_X\al)^{k+1}} \left[ \int_{M_1}^{\lambda_1} e^{-\left(x_{2}-x_{1}\right) \left(\lambda_{1}-\nu_{1}\right)}
\,(\lambda_1-\nu_1)^k d\nu_1\right]\\
\qquad\qquad& \cdot\left[\bl\,\int_{\lambda_3}^{\lambda_2}  e^{-\left(x_{3}-x_{1}\right) \left(\lambda_{2}-\nu_{2}\right)}
(\lambda_2-\nu_2)^{k-1}(\nu_2-\lambda_3)^{k-1} d\nu_2
\right.\\&\qquad\qquad\left.
+\bx\al\int_{\lambda_3}^{\lambda_2}  e^{-\left(x_{3}-x_{1}\right) \left(\lambda_{2}-\nu_{2}\right)}
(\lambda_2-\nu_2)^{k}(\nu_2-\lambda_3)^{k-1} d\nu_2\right]\\
&\asymp \frac{\al^{3k-1}}{(1+\beta_X\al)^{k+1}} \frac{\al^{k+1}}{(1-\ax\al)^{k+1}}
\left[\bl\,\frac{\bl^{2k-1}}{(1+(x_3-x_1)\bl)^{k}}
+\bx\al\frac{\bl^{2k}}{(1+(x_3-x_1)\bl)^{k+1}}\right]\\
&=\frac{\al^{4k}\bl^{2k}(1+(x_3-x_1)\bl+\bx\al)}{(1+\beta_X\al)^{k+1}(1-\ax\al)^{k+1}(1+(x_3-x_1)\bl)^{k+1}}.
\end{align*}

We conclude by considering two cases $(x_3-x_1)\bl\ge\bx\al$ and $(x_3-x_1)\bl\le\bx\al$.

\bigskip

We now consider the case $\bl\ge \al$.

{\bf Step 1.}
We use the  formula \eqref{alpha-} from Lemma
\ref{J1alpha}.
We will estimate the integral
$I(X,\lambda)$ given
by
\begin{align}\label{start231}
I=I(X,\lambda)&\asymp\int_{\lambda_3}^{\lambda_2}\int_{\lambda_2}^{\lambda_1}
\frac{\left(\nu_{1}-\nu_{2}\right) \left(\alpha_{\lambda}-\alpha_{X} \left(\lambda_{1}-\nu_{1}\right) \left(\lambda_{2}-\nu_{2}\right)\right)}{(1-\alpha_X(\nu_1-\nu_2))^{k+1}}
\\ &\qquad\qquad\qquad
\,(\nu_1-\lambda_3) (\nu_2-\lambda_3)
e^{-\left(x_{2}-x_{3}\right) \left(\lambda_{1}-\nu_{1}\right)-\left(x_{3}-x_{1}\right) \left(\nu_{2}-\lambda_{3}\right)}
\,W_k(\lambda,\nu) d\nu_1d\nu_2.\nonumber
\end{align}

If $x_3-x_1\geq x_2 -  x_3$ the proof  is also similar to the previous cases since then $x_2-x_1\asymp x_2-x_1$.

\bigskip

The case $\bl\ge \al$ and $\bpx \le \apx$ remains. The proof is more involved than in preceding cases.

{\bf Step 2.} 
We define $I_1 = \int_{\lambda_3}^{\lambda_2} \int_{M_1}^{\lambda_1} \ldots$ 
and $I_2=I-I_1$.    As usual, we consider instead $\tilde{I}_1$ and $\tilde{I}_2$. We will show that $\tilde{I}_1\geq \tilde{I_2}$. We have, using $\nu_1-\nu_2\asymp \lambda_1-\nu_2$,
\begin{align*}
\tilde{I}_1\geq\int_{(2\,\lambda_1+\lambda_2)/3} ^{(3\,\lambda_1+\lambda_2)/4}
\,[\dots]\gtrsim e^{-\left(x_{2}-x_{3}\right)\al/3} \,\al^{2\,(k-1)+2} \bl^{k}
(\lambda_1-\nu_2)
\frac{\left(1-\alpha_{X}\left(\lambda_{2}-\nu_{2}\right)\right)}{(1-\alpha_X(\lambda_1-\nu_2))^{k+1}}
\end{align*}
while, similarly as in \eqref{est:I2less1}  and \eqref{est:I2more1},
\begin{align}
\tilde{I}_2&\lesssim e^{-\left(x_{2}-x_{3}\right)\al/2} \,\al^{k-1} \bl^{k}
\int_{\lambda_2}^{M_1} 
\frac{(\nu_1-\nu_2) (\nu_1-\lambda_2)^{k-1}}{(1-\alpha_X(\nu_1-\nu_2))^{k+1}}
(\al-\alpha_X\al (\lambda_2-\nu_2))\,d\nu_1\nonumber\\
&\lesssim e^{-\left(x_{2}-x_{3}\right)\al/2} \,\al^{1+(k-1)} \bl^{k}
\int_{\lambda_2}^{M_1} 
\frac{(\nu_1-\nu_2) (\nu_1-\lambda_2)^{k-1}}{(1-\alpha_X(\nu_1-\nu_2))^{k}}\,d\nu_1
\lesssim \frac{\al^{2k}\bl^k(\lambda_1-\nu_2)}{(1-\alpha_X\al)^{k}}.\label{K1}
\end{align}

\begin{align*}
\frac{\tilde{I}_1}{\tilde{I}_2}
&\gtrsim
\frac{e^{\left(x_{2}-x_{3}\right)\al/6}
\left(1-\alpha_{X}\left(\lambda_{2}-\nu_{2}\right)\right)}{(1-\alpha_X(\lambda_1-\nu_2))}
\gtrsim \frac{(1+\left(x_{2}-x_{3}\right)\al/6)
\left(1-\alpha_{X}\left(\lambda_{2}-\nu_{2}\right)\right)}{(1-\alpha_X(\lambda_1-\nu_2))}\\
&\gtrsim \frac{(1/6-\alpha_X\al/6)
\left(1-\alpha_{X}\left(\lambda_{2}-\nu_{2}\right)\right)}{(1-\alpha_X(\lambda_1-\nu_2))}
\geq \frac{(1/6-\alpha_X(\al+\lambda_2-\nu_2)/6)}{(1-\alpha_X(\lambda_1-\nu_2))}=1/6.
\end{align*}
since $x_{2}-x_{3}\asymp -\alpha_X$.

{\bf Step 3. Estimation of $I_1$.} By Step 2, in order to estimate  $I$, it is enough  to estimate $I_1$. We have
\begin{align*}
I\asymp I_1&=\al^{k-1} \bl^k\,\left\lbrace\left[\al
 \int_{M_1}^{\lambda_1} 
e^{-\left(x_{2}-x_{3}\right) \left(\lambda_{1}-\nu_{1}\right)}
\,(\lambda_1-\nu_1)^{k-1}\,d\nu_1
\right]\right.
\\&\qquad\qquad\qquad\left.
\,\left[\int_{\lambda_3}^{\lambda_2} 
e^{-\left(x_{3}-x_{1}\right) \left(\nu_{2}-\lambda_{3}\right)}
\frac{(\lambda_1-\nu_2)^k \left(\lambda_{2}-\nu_{2}\right)^{k-1}(\nu_2-\lambda_3)^{k+1}}{(1-\alpha_X(\lambda_1-\nu_2))^{k+1}}d\nu_2\right]
\right.\\&\qquad\qquad\qquad\left.
-\alpha_X\,\left[\al
 \int_{M_1}^{\lambda_1} 
e^{-\left(x_{2}-x_{3}\right) \left(\lambda_{1}-\nu_{1}\right)}
\,(\lambda_1-\nu_1)^{k}\,d\nu_1
\right]\right.
\\&\qquad\qquad\qquad\left.
\,\left[\int_{\lambda_3}^{\lambda_2} 
e^{-\left(x_{3}-x_{1}\right) \left(\nu_{2}-\lambda_{3}\right)}
\frac{(\lambda_1-\nu_2)^k \left(\lambda_{2}-\nu_{2}\right)^{k}(\nu_2-\lambda_3)^{k+1}}{(1-\alpha_X(\lambda_1-\nu_2))^{k+1}}d\nu_2\right]\right\rbrace\\
&\asymp\al^{k-1} \bl^k\,\left\lbrace\left[\al
 \int_{M_1}^{\lambda_1} 
e^{-\left(x_{2}-x_{3}\right) \left(\lambda_{1}-\nu_{1}\right)}
\,(\lambda_1-\nu_1)^{k-1}\,d\nu_1
\right]\right.
\\&\qquad\qquad\qquad\left.
\,\left[\frac{\bl^{2k-1}}{(1-\alpha_X \bl)^{k+1}}
\int_{\lambda_3}^{M_2} 
e^{-\left(x_{3}-x_{1}\right) \left(\nu_{2}-\lambda_{3}\right)} (\nu_2-\lambda_3)^k\,d\nu_2
\right.\right.\\&\qquad\qquad\qquad\left.\left.
+\bl^k \int_{M_2}^{\lambda_2} 
e^{-\left(x_{3}-x_{1}\right) \left(\nu_{2}-\lambda_{3}\right)}
\frac{(\lambda_1-\nu_2)^k \left(\lambda_{2}-\nu_{2}\right)^{k-1}}{(1-\alpha_X(\lambda_1-\nu_2))^{k+1}}d\nu_2
\right]
\right.\\&\qquad\qquad\qquad\left.
-\alpha_X\,\left[\al
 \int_{M_1}^{\lambda_1} 
e^{-\left(x_{2}-x_{3}\right) \left(\lambda_{1}-\nu_{1}\right)}
\,(\lambda_1-\nu_1)^{k}\,d\nu_1
\right]\right.
\\&\qquad\qquad\qquad\left.
\,\left[\frac{\bl^{2k}}{(1-\alpha_X \bl)^{k+1}}\int_{\lambda_3}^{M_2} 
e^{-\left(x_{3}-x_{1}\right) \left(\nu_{2}-\lambda_{3}\right)}
(\nu_2-\lambda_3)^{k+1}\,d\nu_2
\right.\right.\\&\qquad\qquad\qquad\left.\left.
+\bl^k \int_{M_2}^{\lambda_2} 
e^{-\left(x_{3}-x_{1}\right) \left(\nu_{2}-\lambda_{3}\right)}
\frac{(\lambda_1-\nu_2)^k \left(\lambda_{2}-\nu_{2}\right)^{k}}{(1-\alpha_X(\lambda_1-\nu_2))^{k+1}}d\nu_2\right]\right\rbrace\\
&=\frac{\al^{2k}\bl^k}{(1+(x_2-x_3)\al)^{k}}
\left[\frac{\bl^{3k}}{(1-\alpha_X\bl)^{k+1}(1+(x_3-x_1)\bl)^{k+1}}
+\al^k\bl^k\,U(k-1)+\bl^k\,U(2k-1)\right]
\\&\qquad\qquad\qquad\qquad
+\frac{-\alpha_X\al^{2k}\bl^k}{(1+(x_2-x_3)\al)^{k+1}}
\left[\frac{\bl^{3k+1}}{(1-\alpha_X\bl)^{k+1}(1+(x_3-x_1)\bl)^{k+1}}
+\al^k\bl^k\,U(k)+\bl^k\,U(2k)\right]
\end{align*}
using $(\lambda_1-\nu_2)^k=(\al+\lambda_2-\nu_2)^k\asymp \al^k+(\lambda_2-\nu_2)^k$ and
where 
$$
U(p)= \int_{M_2}^\ll
\frac{e^{\gx(v-\lll)}(\ll-v)^p}
{(1-\ax\al-\ax(\ll-v))^{k+1}}dv.
$$
We then obtain
$$
I(X, \lambda) \asymp  I_1
\asymp c_1(t_1+t_2+t_3) + c_2(T_1+T_2+T_3),
$$
where
\begin{align*}
&c_1= \frac{\al^{2k} \bl^k}{(1+\bx\al)^k},\qquad
c_2 = \frac{-\ax\al^{2k} \bl^k}{(1+\bx\al)^{k+1}},
\\
&t_1= \frac{\bl^{3k}}{(1-\ax\bl)^{k+1}
(1-\gx\bl)^{k+1}}, \qquad
t_2=\al^k\bl^k U(k-1), \qquad
t_3=\bl^k U(2k-1)
\\
&T_1 =\frac{\bl^{3k+1}}{(1-\ax\bl)^{k+1}
(1-\gx\bl)^{k+1}}, \qquad 
T_2=\al^k\bl^k U(k)
, \qquad
T_3=\bl^k U(2k).
\end{align*}

The rest of the proof is straightforward but tedious. We will only give its outline and
omit the details.

In order to estimate the four needed integrals of the form $U(p)$, with 
$p=k-1$, $2k-1$, $k$, $2k$,  it is useful to consider four cases:

\begin{itemize}
\item[(C1)] $-\ax\al \le 1, -\gx\bl\le 1$,

\item[(C2)] $-\ax\al \ge 1, -\gx\bl\le 1$,

\item[(C3)] $-\ax\al \le 1, -\gx\bl\ge 1$,

\item[(C4)] $-\ax\al \ge 1, -\gx\bl\ge 1$.
\end{itemize}

We work separately with $k>1$, $k=1$ and $k< 1$.
In the cases (C3) and (C4), we use the estimate
$$
\int_1^a e^x x^M dx\asymp  e^a a^M,
$$
for $a\geq 2$.  Otherwise, the integrals are easily estimated.  In all cases, one notes that $\al^k U(k-1)\lesssim U(2k-1)$ and $\al^k U(k)\lesssim U(2k)$.

Note that for $p=k$  the estimate of $U(k)$ has a logarithmic factor, which is also the case when $k=1$ for $U(2k-1)=U(1)=U(k)$.  However, these terms are always dominated by other non-logarithmic terms, thanks to the inequality $\ln x\le x^k$, for $x$ large enough.

When $k>1$, after estimating the integrals $U(p)$, we check that $t_1\ge t_2+t_3$ and $T_1\ge T_2+T_3$ so that
$I(\lambda,X)\asymp c_1t_1+c_2T_1$ and the estimate follows by putting the sum on a common denominator.

When $k\le 1$, we consider two following cases.  When $-\gx\al\ge 1$, the proof is similar as for $k>1$.

When $-\gx\al\le 1$, it is easy to see that $t_3\ge t_2$
and that $T_1\ge T_3 \ge T_2$. Next, we check that 
$c_2T_1/ (c_1t_1)\ge 1$ and $c_2T_1/ (c_1t_3)\ge 1$.
Finally, $I(X,\lambda) \asymp c_2T_1$ and the estimate follows.

\subsection{ Weyl chamber $C_{213}$}

If we assume $\al\geq \bl$ and we use formula \eqref{alpha-}, we easily prove the relevant estimate using the same approach as for, say, $C_{321}$.  
If $\bl\geq\al$ and $x_1-x_3\geq 	x_2-x_1$, the same approach works with formula \eqref{beta+} since $x_3-x_1\asymp x_2-x_3$.  We then consider the case $\bl\geq\al$ and $x_1-x_3\leq 	x_2-x_1$ so that 
$x_2-x_1\asymp x_2-x_3$. Using formula \eqref{beta+}, we have
\begin{align*}
I_1
&\asymp \int_{\lambda_3}^{\lambda_2} \int_{M_1}^{\lambda_1} \,e^{-\left(x_{2}-x_{1}\right) \left(\lambda_{1}-\nu_{1}\right)-\left(x_{1}-x_{3}\right) \left(\nu_{2}-\lambda_{3}\right)}
\frac{\left(\nu_{1}-\nu_{2}\right) \left(\bl+\beta_X \left(\nu_{1}-\lambda_{3}\right) \left(\lambda_{2}-\nu_{2}\right)\right)}{(1+\beta_X(\nu_1-\nu_2))^{k+1}}
\\
 &\qquad\qquad\qquad
\,(\lambda_1-\nu_1) (\lambda_1-\nu_2)
\,W_k(\lambda,\nu)\,d\nu_1d\nu_2
\end{align*}
and $I_2=I-I_1$.

Now, since  $\lambda_1-\nu_2 \asymp \nu_1-\nu_2$, we have:
\begin{align*}
\tilde{I_1}&\geq \int_{(2\,\lambda_1+\lambda_2)/3}^{(3\,\lambda_1+\lambda_2)/4}\dots 
\gtrsim \frac{\al^{k+(k-1)+1} \bl^{k-1} e^{-\left(x_{2}-x_{1}\right) \al/3} (\lambda_1-\nu_2)  \left(\bl+\beta_X \bl\left(\lambda_{2}-\nu_{2}\right)\right)}{(1+\beta_X(\lambda_1-\nu_2))^{k+1}}.
\end{align*}

Now we proceed similarly as in \eqref{est:I2less1}
and \eqref{est:I2more1}.
\begin{align}
 \tilde{I_2}&\lesssim\al^{k} \bl^{k-1} e^{-\left(x_{2}-x_{1}\right) \al/2} \int_{\lambda_2}^{M_1} 
\frac{\left(\nu_{1}-\nu_{2}\right)\left(\bl+\beta_X \bl \left(\lambda_{2}-\nu_{2}\right)\right)}{(1+\beta_X(\nu_1-\nu_2))^{k+1}}
(\nu_1-\lambda_2)^{k-1} \,d\nu_1\nonumber\\
&\leq\al^{k} \bl^{k} e^{-\left(x_{2}-x_{1}\right) \al/2} \int_{\lambda_2}^{M_1} 
\frac{\left(\nu_{1}-\nu_{2}\right)}{(1+\beta_X(\nu_1-\nu_2))^{k}}
(\nu_1-\lambda_2)^{k-1} \,d\nu_1\lesssim \frac{\al^{2k} \bl^{k} (\lambda_1-\nu_2) e^{-\left(x_{2}-x_{1}\right) \al/2} }{(1+\beta_X(\lambda_1-\nu_2))^{k}}.\label{K4}
\end{align}

Hence,
\begin{align*}
\frac{\tilde{I_1}}{\tilde{I_2}}
&\gtrsim \frac{e^{\left(x_{2}-x_{1}\right) \al/6}\, \left(1+\beta_X\left(\lambda_{2}-\nu_{2}\right)\right)}{1+\beta_X(\lambda_1-\nu_2)}
\gtrsim \frac{(1+\bx\al/6)\, \left(1+\beta_X \left(\lambda_{2}-\nu_{2}\right)\right)}{1+\beta_X(\lambda_1-\nu_2)}\\
&\geq \frac{1+\bx(\lambda_1-\nu_{2})/6}{1+\beta_X(\lambda_1-\nu_2)}\geq 1/6.
\end{align*}

We then have $I\asymp I_1$ and since $\lambda_1-\nu_2\asymp \nu_1-\nu_2$,
\begin{align*}
I_1
&\asymp \al^{k-1} \bl^{k} \int_{\lambda_3}^{\lambda_2} \int_{M_1}^{\lambda_1} \,e^{-\left(x_{2}-x_{1}\right) \left(\lambda_{1}-\nu_{1}\right)-\left(x_{1}-x_{3}\right) \left(\nu_{2}-\lambda_{3}\right)}
\frac{\left(\lambda_{1}-\nu_{2}\right)^{k+1} \left(1+\beta_X \left(\lambda_{2}-\nu_{2}\right)\right)}{(1+\beta_X(\lambda_1-\nu_2))^{k+1}}
\\
 &\qquad\qquad\qquad
\,(\lambda_1-\nu_1)^k (\lambda_2-\nu_2)^{k-1} (\nu_2-\lambda_3)^{k-1}
\,d\nu_1d\nu_2\\
&=\al^{k-1} \bl^{k} \left[\int_{M_1}^{\lambda_1} \,e^{-\left(x_{2}-x_{1}\right) \left(\lambda_{1}-\nu_{1}\right)}\,(\lambda_1-\nu_1)^k \,d\nu_1\right]
\\&\qquad\cdot
\,\left[ \int_{\lambda_3}^{\lambda_2} \,e^{-\left(x_{1}-x_{3}\right) \left(\nu_{2}-\lambda_{3}\right)}
\frac{\left(\lambda_{1}-\nu_{2}\right)^{k+1} (\lambda_2-\nu_2)^{k-1} (\nu_2-\lambda_3)^{k-1}}{(1+\beta_X(\lambda_1-\nu_2))^{k+1}}
\,d\nu_2
\right.\\&\qquad\qquad+\bx \left.  \int_{\lambda_3}^{\lambda_2} \,e^{-\left(x_{1}-x_{3}\right) \left(\nu_{2}-\lambda_{3}\right)}
\frac{\left(\lambda_{1}-\nu_{2}\right)^{k+1} (\lambda_2-\nu_2)^{k} (\nu_2-\lambda_3)^{k-1}}{(1+\beta_X(\lambda_1-\nu_2))^{k+1}}
\,d\nu_2\right].
\end{align*}

Now, let $J_1= \int_{\lambda_3}^{M_2} \dots$ and  $J_2= \int_{M_2}^{\lambda_2} \dots$ in the first integral involving $\nu_2$.  We have
\begin{align*}
J_1\geq \int_{(3\lambda_3+\lambda_2)/4}^{(2\lambda_3+\lambda_2)/3}  \gtrsim \frac{\bl^{(k+1)+2(k-1)+1} e^{-\left(x_{1}-x_{3}\right) \bl/3}}{(1+\beta_X\bl)^{k+1}}
\end{align*}
while
\begin{align*}
J_2&\lesssim e^{-\left(x_{1}-x_{3}\right) \bl/2} \int_{M_2}^{\lambda_2} 
\frac{\left(\lambda_{1}-\nu_{2}\right)^{k+1} (\lambda_2-\nu_2)^{k} (\nu_2-\lambda_3)^{k-1}}{(1+\beta_X(\lambda_1-\nu_2))^{k+1}}
\,d\nu_2\\
&\lesssim e^{-\left(x_{1}-x_{3}\right) \bl/2} \frac{\left(\lambda_{1}-\lambda_3\right)^{k+1} }{(1+\beta_X(\lambda_1-\lambda_3))^{k+1}} 
\int_{M_2}^{\lambda_2}  (\lambda_2-\nu_2)^{k} (\nu_2-\lambda_3)^{k-1}\,d\nu_2
\asymp \frac{\bl^{3k+1}e^{-\left(x_{1}-x_{3}\right) \bl/2} }{(1+\beta_X\bl)^{k+1}}\lesssim J_1.
\end{align*}

We proceed in the same manner for the second integral in $\nu_2$.

We have
\begin{align*}
I_1
&\asymp \al^{k-1} \bl^{k} \left[\int_{M_1}^{\lambda_1} \,e^{-\left(x_{2}-x_{1}\right) \left(\lambda_{1}-\nu_{1}\right)}\,(\lambda_1-\nu_1)^k \,d\nu_1\right]
\\&\qquad\cdot
\,\left[ \int_{\lambda_3}^{M_2} \,e^{-\left(x_{1}-x_{3}\right) \left(\nu_{2}-\lambda_{3}\right)}
\frac{\left(\lambda_{1}-\nu_{2}\right)^{k+1} (\lambda_2-\nu_2)^{k-1} (\nu_2-\lambda_3)^{k-1}}{(1+\beta_X(\lambda_1-\nu_2))^{k+1}}
\,d\nu_2
\right.\\&\qquad\qquad+\bx \left.  \int_{\lambda_3}^{M_2} \,e^{-\left(x_{1}-x_{3}\right) \left(\nu_{2}-\lambda_{3}\right)}
\frac{\left(\lambda_{1}-\nu_{2}\right)^{k+1} (\lambda_2-\nu_2)^{k} (\nu_2-\lambda_3)^{k-1}}{(1+\beta_X(\lambda_1-\nu_2))^{k+1}}
\,d\nu_2\right]\\
&\asymp   \al^{k-1} \bl^{k} \left[\int_{M_1}^{\lambda_1} \,e^{-\left(x_{2}-x_{1}\right) \left(\lambda_{1}-\nu_{1}\right)}\,(\lambda_1-\nu_1)^k \,d\nu_1\right]
\\&\qquad\cdot
\,\left[\frac{\bl^{(k+1)+(k-1)} }{(1+\beta_X\bl)^{k+1}} + \frac{\bl^{(k+1)+k} }{(1+\beta_X\bl)^{k+1}} \right]\int_{\lambda_3}^{M_2} \,e^{-\left(x_{1}-x_{3}\right) \left(\nu_{2}-\lambda_{3}\right)}
 (\nu_2-\lambda_3)^{k-1}
\,d\nu_2\\
&\asymp \al^{k-1} \bl^{k} \frac{\al^{k+1}}{(1+(x_2-x_1)\al)^{k+1}} 
\,\frac{\bl^{2k}(1+\bx\bl) }{(1+\beta_X\bl)^{k+1}}\frac{\bl^{k}}{(1+(x_3-x_1)\bl)^{k}}
\end{align*}
which gives the correct estimate.


\begin{thebibliography}{99}

\bibitem{Amri01} B.{} Amri.  \textit{On the integral representations for Dunkl kernels of type $A_2$}, Journal of Lie Theory, Vol.{} 26, 2016, 1163--1175.

\bibitem{Anker} J.-P. Anker. \textit{An introduction to Dunkl theory and its analytic aspects}, Analytic, algebraic and geometric aspects of differential equations. Bedlewo, Poland, September 2015, 3--58. Springer International Publishing.

\bibitem{AnkerDH}
J.-P.{} Anker, J. Dziuba\'nski, A. Hejna.
\textit{Harmonic Functions,Conjugate Harmonic Functions and
the Hardy Space $H^1$ in the Rational Dunkl Setting}, 
Journal of Fourier Analysis and Applications (2019) 25:2356-2418.



\bibitem{AnkerDT}
J.-P. Anker, N. Ben Salem, J. Dziuba\'nski,  N. Hamda.
\textit{The Hardy Space $H^1$ in the Rational Dunkl Setting}, Constructive Approximation (2015) 42, 93--128.

\bibitem{AO}
J.-P.{} Anker and P.{} Ostellari. \textit{The heat kernel on noncompact symmetric spaces}, Lie Groups and Symmetric Spaces: In Memory of F.I. Karpelevich, Amer.{} Math.{} Soc.{} (2) Vol.{} 210, 2003.

\bibitem{debie}
H. De Bie, P.  Lian
\textit{Dunkl intertwining operator for symmetric groups}(2021),
arXiv:2009.02087 .


\bibitem{Dunkl} C. F. Dunkl. \textit{Intertwining operators associated to the group $S_3$}, , Trans. Amer. Math.Soc. 347 (1995), 3347--3374.

 \bibitem{Dtalk2017}  J. Dziuba\'nski, \textit{Hardy spaces for certain semigroups of linear operators}, Conference ``Analysis and Application'' in honor of  E. M. Stein,  a talk of September 4, 2017. https://math.uni.wroc.pl/analysis2017
 

 \bibitem{DH} J. Dziuba\'nski, A. Hejna. \textit{Upper and lower bounds for Dunkl heat kernel}, arXiv:2111.03513, 2021.


\bibitem{PGPS0}
P.{} Graczyk and P.{} Sawyer. \textit{Sharp estimates for W-invariant Dunkl and heat kernels in the $A_n$ case}, Bull. Sc. Math., Bulletin des Sciences Math\'ematiques, 186, 103271.

\bibitem{PGPS1} 
P.{} Graczyk and P.{} Sawyer.
\textit{Sharp Estimates of Radial Dunkl and Heat Kernels in the Complex Case $A_n$}, Comptes Rendus, Volume 359, issue 4 (2021), 427--437.

\bibitem{Roesler} R\"osler, M., 1998. \textit{Generalized Hermite polynomials and the heat equation for Dunkl operators}. Communications in Mathematical Physics 192 (3), 519--542.

\bibitem{Sawyer} P.\ Sawyer. \textit{A Laplace-Type Representation of the Generalized Spherical}, 
 Mediterr. J. Math. 14, 147 (2017).
\end{thebibliography}
\end{document}